\title{\vspace{-0.7cm} Directed graphs without short cycles}
\author{
Jacob Fox\thanks{Department of Mathematics, Princeton, Princeton, NJ
08544. Email: {\tt jacobfox@math.princeton.edu}. Research supported
by an NSF Graduate Research Fellowship and a Princeton Centennial
Fellowship.} \and Peter Keevash \thanks{School of Mathematical
Sciences, Queen Mary, University of London, Mile End Road, London E1
4NS, UK. Email: p.keevash@qmul.ac.uk. Research supported in part by
NSF grant DMS-0555755.} \and Benny Sudakov\thanks{Department of
Mathematics, UCLA, Los Angeles, 90095. E-mail:
bsudakov@math.ucla.edu. Research supported in part by NSF CAREER
award DMS-0546523, and a USA-Israeli BSF grant.} }
\documentclass[11pt]{article}
\usepackage{amsfonts,amssymb,amsmath,latexsym}

\newenvironment{proof}
      {\medskip\noindent{\bf Proof.}\hspace{1mm}}
      {\hfill$\Box$\medskip}

\def\qed{\ifvmode\mbox{ }\else\unskip\fi\hskip 1em plus 10fill$\Box$}

\newtheorem{theorem}{Theorem}[section]

\newtheorem{lemma}[theorem]{Lemma}

\newtheorem{corollary}[theorem]{Corollary}

\newcommand{\nib}[1]{\noindent {\bf #1}}

\newcommand{\sm}{\setminus}

\newcommand{\eps}{\epsilon}

\newcommand{\hg}{\alpha}

\def\COMMENT#1{}

\begin{document}
\date{}

\maketitle

\begin{abstract}
For a directed graph $G$ without loops or parallel edges, let $\beta(G)$
denote the size of the smallest feedback arc set, i.e., the smallest
subset $X \subset E(G)$ such that $G \sm X$ has no directed cycles.
Let $\gamma(G)$ be the number of unordered pairs of vertices of $G$
which are not adjacent. We prove that every directed graph whose
shortest directed cycle has length at least $r \ge 4$ satisfies
$\beta(G) \le c\gamma(G)/r^2$, where $c$ is an absolute constant.
This is tight up to the constant factor and extends a result of
Chudnovsky, Seymour, and Sullivan.
\COMMENT{Tournaments show we need to assume $r \ge 4$.}

This result can be also used to answer a question of Yuster
concerning almost given length cycles in digraphs. We show that for
any fixed $0 < \theta < 1/2$ and sufficiently large $n$, if $G$ is a
digraph with $n$ vertices and $\beta(G) \ge \theta n^2$, then for
any $0 \le m \le \theta n-o(n)$ it contains a directed cycle
whose length is between $m$ and $m+6 \theta^{-1/2}$. Moreover, there
is a constant $C$ such that either $G$ contains directed cycles of
every length between $C$ and $\theta n-o(n)$ or it is close to a
digraph $G'$ with a simple structure: every strong component of $G'$
is periodic. These results are also tight up to the constant
factors.
\end{abstract}

\section{Introduction}

A digraph (directed graph) $G$ is a pair $(V_G,E_G)$ where $V_G$ is a finite set of vertices and $E_G$ is a set of ordered
pairs $(u,v)$ of vertices called edges. All digraphs we consider in this paper are simple, i.e.,
they do not have loops or parallel edges. A path of length $r$ in $G$ is a collection of distinct vertices $v_1,\ldots,v_r$
together with edges $(v_i,v_{i+1})$ for $1 \leq i \leq r-1$. Moreover, if $(v_r,v_1)$ is also an edge,
then it is an $r$-cycle.

The concept of cycle plays a fundamental role in graph theory, and
there are numerous papers which study cycles in graphs. In contrast,
the literature on cycles in directed graphs is not so extensive. It
seems the main reason for this is that questions concerning cycles
in directed graphs are often much more challenging than the
corresponding questions in graphs. An excellent example of this
difficulty is the well-known Caccetta-H\"aggkvist conjecture
\cite{CaH}. For $r \ge 2$, we say that a digraph is {\em $r$-free}
if it does not contain a directed cycle of length at most $r$. The
Caccetta-H\"aggkvist conjecture states that every $r$-free digraph
on $n$ vertices has a vertex of outdegree less than $n/r$. This
notorious conjecture is still open even for $r=3$, and we refer the
interested reader to the recent surveys \cite{N,S2}, which discuss
known results on this problem and other related open questions.

In approaching the Caccetta-H\"aggkvist conjecture it is natural to
see what properties of an $r$-free digraph one can prove. A {\em
feedback arc set} in a digraph is a collection of edges whose
removal makes the digraph acyclic. For a digraph $G$, let $\beta(G)$
denote the size of the smallest feedback arc set. This parameter
appears naturally in testing of electronic circuits and in efficient
deadlock resolution (see, e.g., \cite{LS,Sh}). It is also known that
it is NP-hard to compute the minimum size of a feedback arc set even
for tournaments \cite{Al,CTY} (a tournament is an oriented complete
graph). Let $\gamma(G)$ be the number of unordered pairs of vertices
of $G$ which are not adjacent. Chudnovsky, Seymour and Sullivan
\cite{CSS} conjectured that if $G$ is a $3$-free digraph then
$\beta(G)$ is bounded from above by $\gamma(G)/2$. They proved this
conjecture in two special cases, when the digraph is the union of
two cliques or is a circular interval digraph. Moreover, for general
$3$-free digraphs $G$, they showed that $\beta(G) \le \gamma(G)$.

Generalizing this conjecture, Sullivan \cite{S}
suggested that every $r$-free digraph G satisfies $\beta(G) \le
2\gamma(G)/(r+1)(r-2)$, and gave an example showing that this would
be best possible. She posed an open problem to prove that $\beta(G)
\le f(r)\gamma(G)$ for every $r$-free digraph $G$, for some function
$f(r)$ tending to $0$ as $r \to \infty$. Here we establish a
stronger bound which shows that Sullivan's conjecture is true up to
a constant factor. This extends the result of Chudnovsky, Seymour
and Sullivan to general $r$.

\begin{theorem} \label{beta-gamma}
For $r \ge 3$, every $r$-free digraph $G$ satisfies $\beta(G) \le
800\gamma(G)/r^2$.
\end{theorem}

The above result is tight up to a constant factor. Indeed, consider
a blowup of an $(r+1)$-cycle, obtained by taking disjoint sets
$V_1,\cdots,V_{r+1}$ of size $n/(r+1)$ and all edges from $V_i$ to
$V_{i+1}$, $1 \le i \le r+1$ (where $V_{r+2}=V_1$). This digraph on
$n$ vertices is clearly $r$-free, has $\gamma(G)={n \choose
2}-\frac{n^2}{r+1} \geq \frac{n(n-2)}{4}$, and $\beta(G) \geq
\frac{n^2}{(r+1)^2}$. Indeed, $G$ contains $\frac{n^2}{(r+1)^2}$
edge-disjoint cycles of length $r+1$, and one needs to delete at
least one edge from each cycle to make $G$ acyclic.

In order to prove Theorem \ref{beta-gamma}, we obtain a bound on the
edge expansion of $r$-free digraphs which may be of independent
interest. For vertex subsets $S, T \subset V_G$, let $e_G(S,T)$ be
the number of edges in $G$ that go from $S$ to $T$. The {\em edge
expansion} $\mu(S)$ of a vertex subset $S \subset V_G$ with
cardinality $|S| \leq |V_G|/2$ is defined to be
$$\frac{1}{|S|} \min \big\{ e_G(S,V_G \sm S), e_G(V_G \sm S,S) \big\}.$$
The edge expansion $\mu=\mu(G)$ of $G$ is the minimum of $\mu(S)$
over all vertex subsets $S$ of $G$ with $|S| \leq |V_G|/2$. We show
that $r$-free digraphs can not have large edge expansion.

\begin{theorem}\label{mu}
Suppose $G$ is a digraph on $n$ vertices, $r \ge 9$ and
$\mu=\mu(G) \geq 25n/r^2$. Then every vertex of $G$ is contained in
a directed cycle of length at most $r$.
\end{theorem}

Using this result, it is easy to deduce the following corollary,
which implies Theorem \ref{beta-gamma} in the case $G$ is not too
dense.

\begin{corollary} \label{beta}
Every $r$-free digraph $G$ on $n$ vertices satisfies $\beta(G) \le 25n^2/r^2$.
\COMMENT{Trivial for $r \le 5$, say.}
\end{corollary}

Corollary \ref{beta} will also enable us to answer the following
question posed by Yuster \cite{Y}. Suppose that a digraph $G$ on $n$
vertices is far from being acyclic, in that $\beta(G) \ge \theta
n^2$. What lengths of directed cycles can we find in $G$? Yuster
\cite{Y} showed that for any $\theta>0$ there are constants $K$ and
$\eta$ so that for any $m \in (0,\eta n)$ there is a directed cycle
whose length is between $m$ and $m+K$. He gave
examples showing that one must have $K \ge \theta^{-1/2}$ and $\eta
\le 4\theta$, and posed the problem of determining the correct order
of magnitude of these parameters as a function of $\theta$. The
following theorem, which is tight up to constant factors for both $K$ and $\eta$, answers Yuster's question.

\begin{theorem} \label{given-length}
For any $0<\delta,\theta<1$ the following holds for $n$ sufficiently large.
Suppose $G$ is a digraph on $n$ vertices with $\beta(G) \ge \theta n^2$.
Then for any $0 \le m \le (1-\delta)\theta n$
there is $m \le \ell \le m + (5 + \delta) \theta^{-1/2}$
such that $G$ contains a directed cycle of length $\ell$.
\end{theorem}

Moreover, we can show that $G$ either contains directed cycles of all lengths between
some constant $C$ and $\theta n-o(n)$ or is highly structured in the following sense.
Say that $G$ is {\em periodic} if the length of every directed cycle in $G$ is divisible by
some number $p \ge 2$, and {\em pseudoperiodic} if every strong component $C$ is periodic
(possibly with differing periods). A digraph is {\em strong} if, for every pair $u,v$ of vertices, there is a path from
$u$ to $v$ and a path from $v$ to $u$. A {\em strong component} of a digraph $G$ is a maximal strong subgraph of $G$.
 A pseudoperiodic digraph $G$ is highly structured, as
Theorem 10.5.1 of \cite{BJG} shows that a strongly connected digraph with period $p$
is contained in the blowup of a $p$-cycle. Let $\lambda(G)$ denote the
minimum number of edges of $G$ that need to be deleted from $G$ to obtain a pseudoperiodic digraph. Note that
$\beta(G) \ge \lambda(G)$, as every acyclic digraph is pseudoperiodic.

\begin{theorem} \label{periodic}
For any $0<\delta,\theta<1$ there are numbers $C$ and $n_0$ so that
the following holds for $n \ge n_0$.
If $G$ is a digraph on $n$ vertices with $\lambda(G) \ge \theta n^2$
then $G$ contains a directed cycle of length $\ell$ for
any $C \le \ell \le (1-\delta)\theta n$.
\end{theorem}

The rest of this paper is organised as follows. In the next section
we collect two simple lemmas concerning nearly complete digraphs. We need these lemmas in Section 3
to prove Theorems \ref{beta-gamma}, \ref{mu} and Corollary \ref{beta}. In Section 4,
we discuss Szemer\'edi's Regularity Lemma for
digraphs and some of its consequences. We use these results together with Corollary \ref{beta} in Section 5 to prove
Theorems \ref{given-length} and \ref{periodic}. The final section contains some concluding remarks.

\medskip

\nib{Notation.} An oriented graph is a digraph which can be obtained
from a simple undirected graph by orienting its edges. Note that for
$r \geq 2$, every $r$-free digraph is an oriented graph, as two opposite edges on the same pair of
vertices form a $2$-cycle. Suppose $G$ is an oriented graph and $S$
and $T$ are subsets of its vertex set $V_G$. Let $E_G(S,T)$ be the set of edges in $G$ that go
from $S$ to $T$, so $e_G(S,T)=|E_G(S,T)|$. We drop the subscript $G$
if there is no danger of confusion. Let $G[S]$
denote the restriction of $G$ to $S$, in which the vertex set is $S$
and the edges are all those edges of $G$ with both endpoints in $S$, and let
$G \sm S = G[V_G \sm S]$ be the restriction of $G$ to the
complement of $S$. We use the notation $0 < \alpha
\ll \beta$ to mean that there is a increasing function $f(x)$ so
that the following argument is valid for $0 < \alpha < f(\beta)$.
Throughout the paper, we systematically omit floor and ceiling signs
whenever they are not crucial, for the sake of clarity of
presentation. We also do not make any serious attempt to optimize
absolute constants in our statements and proofs.

\section{Basic facts}

We start with two simple lemmas concerning oriented graphs that are
nearly complete. First we prove a lemma which shows that such an
oriented graph contains a vertex that has large indegree and large
outdegree. Consider an oriented graph $G$ whose vertex set is
partitioned $V_G = V_1 \cup V_2$ with $|V_1| = |V_2| = n/2$, such
that all edges go from $V_1$ to $V_2$, and the restriction of $G$ to
each $V_i$ is regular with indegree and outdegree of every vertex
equal to $(1-2\eps)n/4$. The number of edges in $G$ is $(1 -
\eps)n^2/2$ and no vertex has indegree and outdegree both more than
$(1-2\eps)n/4$. This example demonstrates tightness of the following
lemma.

\begin{lemma} \label{in-out}
Let $G$ be an oriented graph with $n$ vertices and $(1 - \eps)n^2/2$ edges.
Then $G$ contains a vertex with indegree and outdegree at least $(1-2\eps)n/4$.
\end{lemma}

\begin{proof}
Suppose for a contradiction that no vertex of $G$ has indegree and
outdegree at least $(1-2\eps)n/4$. Delete vertices one by one whose
indegree and outdegree in the current oriented graph are both less
than $(1-2\eps)n/4$. Let $G'$ be the oriented graph that remains and
$\alpha n$ be the number of deleted vertices. Then $G'$ has
$(1-\alpha)n$ vertices, at least $(1-\eps)n^2/2 - \alpha n \cdot
2(1-2\eps)n/4$ edges, and every vertex has either indegree or
outdegree at least $(1-2\eps)n/4$, but not both. Partition $V_{G'} =
V_1 \cup V_2$, where $V_1$ consists of those vertices of $G'$ that
have indegree at least $(1-2\eps)n/4$. Since
$|V_1|+|V_2|=(1-\alpha)n$ we have $|V_1||V_2| \le (1 - \alpha)^2
n^2/4$, and so
$$e(V_1) + e(V_2) \ge (1 - \eps) n^2/2 - (1-2\eps)\alpha n^2/2 - |V_1||V_2|
\ge (1-2\eps+4\alpha\eps-\alpha^2)n^2/4.$$ We may assume without
loss of generality that $e(V_1)/e(V_2) \geq |V_1|/|V_2|$ (the other case can be treated similarly). In the first
case,
$$e(V_1) \ge \frac{|V_1|}{|V_1|+|V_2|}\left(e(V_1)+e(V_2)\right)
\ge \frac{|V_1|}{|V_1|+|V_2|}\big(1-2\eps+4\alpha\eps-\alpha^2\big)\frac{n^2}{4}
= |V_1|\big(1-2\eps+4\alpha\eps-\alpha^2\big)\frac{n}{4(1-\alpha)}.$$ Then
the average outdegree of a vertex in $V_1$ is at least $(1 - 2\eps +
4\alpha \eps - \alpha^2)\frac{n}{4(1 - \alpha )}$. It is easy to
check as a function of $\alpha$ this is increasing for $\alpha \in
[0, 1)$ and is therefore minimized when $\alpha  = 0$. Therefore the
average outdegree of a vertex in $V_1$ is at least $(1 - 2\eps)n/4$.
Now we can choose a vertex in $V_1$ with outdegree at least the
average, and then by definition of $V_1$ it has both indegree and
outdegree at least $(1 - 2\eps)n/4$, a contradiction.
\end{proof}

We can use this lemma to find in a nearly complete oriented graph a
vertex of very large total degree and reasonably large indegree and
outdegree.

\begin{lemma} \label{total}
Let $G$ be an oriented graph with $n \ge 20$ vertices and $\gamma =
\hg n^2$ non-adjacent pairs, with $\hg \le 1/16$. Then $G$ has a
vertex $v$ of total degree at least $(1 - 4\hg)n$ and indegree and
outdegree at least $\frac{n}{10}$.
\end{lemma}

\begin{proof}
Let $V'$ be those vertices of $G$ with total degree at least $(1 -
4\hg)n$. Then $V \sm V'$ is incident to at least $|V \sm V'|4\hg
n/2$ non-adjacent pairs, so $(n-|V'|)2 \hg n \le \gamma = \hg n^2$,
i.e., $|V'| \ge n/2$. Write $|V'|=\omega n$. The number of edges in
the restriction $G[V'$] of $G$ to $V'$ is at least
$$\binom{|V'|}{2} - \big(\gamma - |V \sm V'|4\hg n/2\big)
= \big(1 - (4\omega-2)\hg/\omega^2 - 1/|V'|\big)|V'|^2/2.$$
Applying Lemma \ref{in-out} to $G[V']$, with
$\eps = (4\omega-2)\hg/\omega^2 + 1/|V'|$,
we find a vertex with indegree and outdegree at least
$$(1-2\eps)|V'|/4 = \big(1/4 - (2\omega-1)\hg/\omega^2\big) \omega n - 1/2 \ge n/8 - 1/2 \ge n/10,$$
where we use the fact that, for fixed $\alpha \le 1/16$, the minimum
of $f(\omega)=\omega/4 + (1-2\omega)\hg/\omega$ for $\omega \in
[1/2,1]$ occurs at $\omega=1/2$. Indeed, for $\omega \geq 1/2$,
$f'(\omega)=\frac{1}{4}-\frac{\alpha}{\omega^2} \geq 0$ and $f(\omega)$ is an
increasing function.
\end{proof}

\section{Finding short cycles}

We will prove Theorem \ref{beta-gamma} by proving that an $r$-free
 digraph can not have large edge expansion. Recall that the
edge expansion $\mu(S)$ of a set $S$ of vertices of a digraph $G$
with cardinality $|S| \leq |V_G|/2$ is defined to be
$$\frac{1}{|S|} \min \big\{ e(S,V_G \sm S), e(V_G \sm S,S)
\big\},$$ and the edge expansion $\mu=\mu(G)$ of $G$ is the minimum
of $\mu(S)$ over all subsets $S \subset V_G$ with $|S| \leq
|V_G|/2$.

Consider a digraph $G$ on $n$ vertices and any vertex $v$ of $G$.
We say that a vertex $w$ has {\em outdistance} $i$ from $v$ if the length of the shortest directed path from
$v$ to $w$ is $i$. ({\em Indistance} is similarly defined.)
Let $N_i$ be the vertices at outdistance exactly $i$ from $v$
and $M_i = \cup_{j \le i} N_i$ the vertices at outdistance at most $i$ from $v$.
It follows from these definitions that any edge from $M_i$ to $V_G \sm M_i$
is in fact an edge from $N_i$ to $N_{i+1}$. We deduce that
$$\mu(M_i)|M_i| \le e(M_i,V_G \sm M_i) = e(N_i,N_{i+1}) \le |N_i||N_{i+1}|.$$
Then the Arithmetic Mean - Geometric Mean Inequality gives
\begin{equation} \label{am-gm}
|N_i| + |N_{i+1}| \ge 2\sqrt{\mu(M_i)|M_i|}.
\end{equation}

The first step of the proof of Theorem \ref{beta-gamma} is Theorem \ref{mu}, which shows that large edge expansion
implies short cycles, and moreover we can find a short cycle through any
specified vertex.

\noindent {\bf Proof of Theorem \ref{mu}.}
Let $v$ be any vertex of $G$. As before, let $N_i$ be the vertices of outdistance exactly $i$ from $v$ and $M_i$
the vertices of outdistance at most $i$ from $v$. Also, let
$a_i=(|N_i|+|N_{i+1}|)/\mu$ and $b_i=\sum_{1 \leq j \leq i} a_j$.
Then $b_{i-1}\mu=2|M_i|-|N_1|-|N_{i}| \leq 2|M_i|$, so dividing both sides of inequality (\ref{am-gm}) by $\mu$
and using $\mu(M_i) \geq \mu$ gives $$a_i =(|N_i|+|N_{i+1}|)/\mu \geq 2\sqrt{\frac{\mu(M_i)}{\mu}\frac{|M_i|}{\mu}} \geq 2\sqrt{|M_i|/\mu} \geq \sqrt{2b_{i-1}}.$$
Adding $b_{i-1}$ to both sides we have $b_i \ge b_{i-1} + \sqrt{2b_{i-1}}$.
Note that $b_1 = a_1 \geq |N_1|/\mu \geq 1$, as otherwise $|N_1| <\mu$ and taking $S=\{v\}$
we have $ \mu(S) \leq |N_1| <\mu$, contradicting the definition of $\mu$. Now we prove by induction that $b_i \geq \frac{2}{5}i^2$.
This is easy to check for $i < 6$ using a calculator and $b_1 \ge 1$. For $i \ge 6$, the induction step is
$$b_i \geq b_{i-1}+\sqrt{2b_{i-1}} \ge
\frac{2}{5} (i-1)^2+\sqrt{4/5}(i-1) \ge \frac{2}{5}i^2.$$
Applying this with $i = \lfloor r/2 \rfloor$ we have
$|M_i| \ge \mu b_{i-1}/2 \ge \mu(i-1)^2/5 > n/2$,
since $\mu \ge 25n/r^2$ and $r \ge 9$.
The same argument shows that there are more than $n/2$ vertices at
indistance at most $i$ from $v$. Therefore there is a vertex at indistance
and outdistance at most $i$ from $v$, which gives a directed cycle
through $v$ of length at most $r$.\COMMENT{Consider the shortest closed walk.}
\qed

Next we deduce Corollary \ref{beta}, which implies our main theorem in
the case when $G$ is not almost complete.

\nib{Proof of Corollary \ref{beta}.}
We suppose that $G$ is $r$-free and prove that $\beta(G) \le 25n^2/r^2$.

First we deal with the case $r \le 10$. In any linear
ordering of the vertices of $G$, deleting the forward edges or the backwards edges makes the digraph acyclic. Since the
number of edges in $G$ is ${n \choose 2}-\gamma(G)$
we have $\beta(G) \le \frac{1}{2}({n \choose 2}-\gamma(G)) < n^2/4$. Hence,
$\beta(G) < 25n^2/r^2$ if $r \le 10$.\COMMENT{Could also use $\beta \le \gamma$ to get $\beta < n^2/6$.}

Next, for $r \ge 11$ we use induction on $n$. Note that if $n \le r$ then $G$ is acyclic and $\beta(G)=0$,
so we can assume that $n>r$. By Theorem \ref{mu} and definition of $\mu$ we can find a set $S$
with $|S|=s \le n/2$ and $\mu(S)=\mu < 25n/r^2$. Note that a digraph formed by taking the disjoint union of two
acyclic digraphs and adding some edges from the first acyclic digraph to the second acyclic digraph is acyclic.
Therefore, using the inequality $n \le 2(n-s)$, we obtain
$$\hspace{1.3cm}\beta(G) \leq \beta(G[S])+\beta(G \sm S) + \mu s
\le 25s^2/r^2 + 25(n-s)^2/r^2 + 25n/r^2 \cdot s \le 25n^2/r^2\,.\hspace{1.3cm} \Box$$

We need one more lemma before the proof of the main theorem, showing that
an $r$-free oriented graph has a linear-sized subset $S$ with small edge expansion.

\begin{lemma} \label{S}
Suppose $r \ge 15$, $0 \le \hg \le 1/16$ and
$G$ is an $r$-free oriented graph on $n \ge 20$ vertices
with $\gamma = \hg n^2$ non-adjacent pairs.
Then there is $S \subset V(G)$
with $n/10 \le |S| \le n/2$ and $\mu(S) < 1500 \hg^2 n/r^2$.
\end{lemma}

\begin{proof}
By Lemma \ref{total} there is a vertex $v$ of total degree at least $(1 - 4\hg)n$
and indegree and outdegree at least $n/10$. As before, let $N_i$ be the vertices of outdistance exactly $i$ from $v$ and $M_i$
the vertices of outdistance at most $i$ from $v$. Since $G$ is $r$-free there is no vertex at indistance
and outdistance at most $\lfloor r/2 \rfloor$ from $v$, so
we can assume without loss of generality
that $|M_i| \le n/2$ for all $i \le \lfloor r/2 \rfloor$.
Also, by choice of $v$ we have $|M_i| \ge |N_1| \ge n/10$,
so we are done if we have $\mu(M_i) < 1500 \hg^2 n/r^2$
for some $i \le \lfloor r/2 \rfloor$. Suppose for a
contradiction that this is not the case. Then equation (\ref{am-gm}) gives
$$|N_i| + |N_{i+1}| \ge 2\sqrt{1500 \hg^2 n/r^2 \cdot n/10} > 24 \hg n/r.$$
Let $s = \lceil \frac{r-5}{4} \rceil \ge r/6$, so $2s+1 \le r/2$.
The above inequality gives
$$|M_{2s+1}|-|N_1| = (|N_2|+|N_3|) + \cdots + (|N_{2s}|+|N_{2s+1}|)
> s \cdot 24 \hg n/r \ge 4\hg n.$$
Let $I_1$ denote the inneighbourhood of $v$. By choice of $v$ we have
$|I_1|+|N_1| \ge (1 - 4\hg)n$, and so $|I_1|+|M_{2s+1}|>n$, and hence
there is a vertex in both $I_1$ and $M_{1+2s}$. This gives a cycle
of length at most $2+2s \le r$, contradiction.
\end{proof}

\nib{Proof of Theorem \ref{beta-gamma}.}
We use induction on $n$ to prove that every digraph $G$ on $n$ vertices satisfies
\begin{eqnarray}\label{inductbound} \beta(G) & \le & 800r^{-2}(\gamma(G) - \gamma(G)^2/n^2).\end{eqnarray}
Note that the right hand side of (\ref{inductbound}) is at least $400\gamma(G)/r^2$ and at most $800\gamma(G)/r^2$
as $0 \le \gamma(G) \le {n \choose 2} \le n^2/2$.
We can assume that $\gamma(G) < n^2/16$, since otherwise we can apply Corollary \ref{beta}
to get $\beta(G) \le 25n^2/r^2 \le 400\gamma(G)/r^2$.
We can also assume that $r \ge 21$, as otherwise $r \le 20$ and
we can use the result of Chudnovsky, Seymour, and Sullivan \cite{CSS} that $3$-free graphs $G$ satisfy
$\beta(G) \le \gamma(G) \le 400 \gamma(G)/r^2$.
Then we can assume that $n \ge 22$, as otherwise $n \leq r$, $G$ is acyclic, and $\beta(G)=0$.
\COMMENT{Used $r \ge 3$ to apply CSS.}

Let $S$ be the set given by Lemma \ref{S}, $G_1=G[S]$,
$G_2 = G \sm S$ and $n_i = |V(G_i)|$, $\gamma = \gamma(G)$, $\gamma_i = \gamma(G_i)$ for $i=1,2$,
so that $n_1+n_2=n$ and $\gamma_+ := \gamma_1+\gamma_2 \le \gamma$.
By choice of $S$ we have $\mu(S)|S| < 1600\gamma^2n_1/n^3r^2$. By deleting all edges from $S$ to $V_G \setminus S$ or
all edges from $V_G \setminus S$ to $S$, we get by the induction hypothesis that
$$\beta(G) \le \beta(G_1) + \beta(G_2) + \mu(S)|S|
\le 800r^{-2}(\gamma_1 - \gamma_1^2/n_1^2 + \gamma_2 - \gamma_2^2/n_2^2 + 2\gamma^2n_1/n^3).$$
Now the Cauchy-Schwartz inequality gives
$\gamma_+^2 = (n_1 \cdot \gamma_1/n_1 + n_2 \cdot \gamma_2/n_2)^2 \le
(n_1^2+n_2^2)(\gamma_1^2/n_1^2 + \gamma_2^2/n_2^2)$
so we have
$$\beta(G) \le 800r^{-2}(\gamma_+ - \gamma_+^2/(n_1^2+n_2^2) + 2\gamma^2n_1/n^3)
\le 800r^{-2}(\gamma - \gamma^2/(n_1^2+n_2^2) + 2\gamma^2n_1/n^3).$$
Here we used $\gamma_+ \le \gamma < n^2/16$
and $n_1^2+n_2^2 \ge \frac{1}{2}(n_1+n_2)^2 = n^2/2$, which give the inequality
$$\gamma - \frac{\gamma^2}{n_1^2+n_2^2} - \gamma_+ + \frac{\gamma_+^2}{n_1^2+n_2^2}
= (\gamma-\gamma_+) \left(1 - \frac{\gamma_+ + \gamma}{n_1^2+n_2^2} \right) \ge 0.$$
Now the desired bound on $\beta(G)$ follows from the inequality
$\gamma^2/(n_1^2+n_2^2) - 2\gamma^2n_1/n^3 \ge \gamma^2/n^2$.
Set $n_1 = tn$, where $1/10 \le t \le 1/2$ by choice of $S$.
It is required to show that $f(t) = \frac{1}{1+2t} - t^2 - (1-t)^2 \ge 0$.
By computing $f'(t)=2-4t-\frac{2}{(1+2t)^2}$ and $f''(t) = \frac{8}{(1+2t)^3} - 4$, we see that for $t \geq 0$,
 $f''$ is a decreasing function and $f''(0) > 0 > f''(1/2)$. Hence
$f'$ increases from $f'(0)=0$ to a maximum and then decreases to $f'(1/2) < 0$, being first nonnegative
until some $t_0 < 1/2$ and then negative afterwards. Therefore, $f$ increases from $f(0)=0$ to a maximum $f(t_0)$ and then
decreases to $f(1/2)=0$ staying nonnegative in the whole interval. This completes the proof. \qed

\section{Regularity}

For our second topic in the paper we will use the machinery of Szemer\'edi's Regularity Lemma,
which we will now describe. We will be quite brief, so for more details and motivation
we refer the reader to the survey \cite{KSi}. First we give some definitions.
The density of a bipartite graph $G = (A,B)$
with vertex classes $A$ and $B$ is defined to be $d_G(A,B) := \frac{e_G(A,B)}{|A||B|}$.
We write $d(A,B)$ if this is unambiguous.
Given $\eps>0$, we say that $G$ is $\eps$-regular if
for all subsets $X\subseteq A$ and $Y\subseteq B$
with $|X|>\eps|A|$ and $|Y|>\eps|B|$ we have that $|d(X,Y)-d(A,B)|<\eps$.
Given $d\in [0,1]$ we say that $G$ is $(\eps,d)$-super-regular
if it is $\eps$-regular and furthermore $d_G(a)\ge (d-\eps) |B|$ for all $a\in A$
and $d_G(b)\ge (d-\eps)|A|$ for all $b\in B$. If $A$ and $B$ are disjoint vertex subsets of a digraph $G$, we say that the pair $(A,B)_G$ is
$\epsilon$-regular if the bipartite graph with vertex sets $A$ and $B$ and edge set $E_G(A,B)$ is $\epsilon$-regular. Similarly,
we say that $(A,B)_G$ is $(\epsilon,d)$-super-regular if the bipartite graph with vertex sets $A$ and $B$ and edge set $E_G(A,B)$
is $(\epsilon,d)$-super-regular.

The Diregularity Lemma is a version of the Regularity Lemma for digraphs due to
Alon and Shapira \cite{ASh} (with a similar proof to the undirected version
of Szemer\'edi).

\begin{lemma}[Diregularity Lemma]\label{direg}
For every $\eps \in (0,1)$ and $M'>0$ there are numbers $M$ and $n_0$
such that if $G$ is a digraph on $n\ge n_0$ vertices,
then there is a partition of the vertices of $G$ into
$V_0,V_1,\cdots,V_k$ for some $M' \le k \le M$
such that $|V_0|\leq \eps n$, $|V_1|=\cdots=|V_k|$
and for all at but at most $\eps k^2$ ordered pairs $1 \le i < j \le k$
the underlying graph of $E_G(V_i,V_j)$ is $\eps$-regular.
\end{lemma}

Given $0 \le d \le 1$ we define the reduced digraph $R$ with parameters
$(\eps,d)$ to have vertex set $[k]=\{1,\cdots,k\}$ and an edge
$ij$ if and only if the underlying graph of $E_G(V_i,V_j)$ is $\eps$-regular
with density at least $d$. Note that if $\eps$ and $d$ are small, $M'$ is large, and $G$ is a dense digraph,
then most edges of $G$ belong to pairs $E_G(V_i,V_j)$ for some edge $ij \in R$.
Indeed, the exceptions are at most $\eps n^2$ edges incident to $V_0$,
at most $n^2/M'$ edges lying within some $V_i$,
at most $\eps n^2$ edges belonging to pairs $E_G(V_i,V_j)$ that are not $\eps$-regular,
and at most $d n^2$ edges belonging to $E_G(V_i,V_j)$ of density less than $d$:
this gives a total less than $2dn^2$ if say $1/M' < \eps \ll d$.
We also need the following path lemma.

\begin{lemma} \label{path}
For every $0<d<1$ there is $\eps_0>0$ so that the following holds for $0<\eps<\eps_0$.
Let $p,n$ be positive integers with $p \geq 4$,
$U_1,\ldots,U_p$ be pairwise disjoint sets of size $n$ and suppose
$G$ is a digraph on $U_1 \cup \cdots \cup U_p$ such that each
$(U_i,U_{i+1})_G$ is $(\eps,d)$-super-regular. (Here, $U_{p+1} :=
U_1$.) Take any $x \in U_1$ and any $y \in U_p$. Then for any $1 \le \ell \le n$
there is a path $P$ in $G$ of length $p\ell$, starting with $x$ and ending with $y$,
in which for every vertex $v \in U_i$, the successor of $v$ on $P$
lies in $U_{i+1}$.
\end{lemma}

This lemma can be easily deduced from the blowup lemma of Koml\'os, Sark\"ozy and Szemer\'edi
(despite $p$ being arbitrary), as shown in \cite{CKKO}. For the sake of completeness and the convenience of
the reader we include the proof here. In fact, for our purposes it is sufficient to apply the result
with $1 \le \ell \le (1-\eps)n$; in that case it is not too hard to prove it directly
with a random embedding procedure, but we omit the details.
Note also that by applying the lemma when $yx$ is an edge we can obtain a directed
cycle of length $p\ell$ for any $1 \le \ell \le n$.

The requirement that $p \ge 4$ in Lemma \ref{path} is necessary. Indeed, if $p=2$ or $p=3$,
 there may not be a path of length $p$ from $x$ to $y$. It is not difficult to show using Lemma \ref{path}
that even in this case we can find a path from $x$ to $y$ of length $p\ell$ for all $2 \le \ell \le n$.
It is even easier to show that we can greedily find such paths for all $2 \le \ell \le dn/2$, and since
this will be sufficient for our purposes, we do so now. In the following argument, if $i$ does not satisfy $1 \leq i \leq p$, then
we define $U_i:=U_j$ with $1 \leq j \le p$ and $i \equiv j \pmod p$. Since each pair $(U_i,U_{i+1})_G$ is
$(\epsilon,d)$-super-regular, each vertex in $U_i$ has at least $(d-\epsilon)n$ outneighbours in $U_{i+1}$, and
we can greedily find a path $P'=v_1 \cdots v_{p\ell-3}$ with starting point $v_1=x$ and with
each $v_i$ in $U_i$, as each such path only contains at most $\ell \leq dn/2$ vertices in each $U_i$.
Let $X$ be the outneighbours of $v_{p\ell-3}$ in $U_{p-2} \setminus P'$
and let $Y$ be the inneighbours of $y$ in $U_{p-1} \setminus P'$, so
$|X| \ge (d-\epsilon)n-\ell \ge (\frac{d}{2}-\epsilon)n \ge \epsilon n$ and similarly $|Y| \ge \epsilon n$. Since the pair
$(U_{p-2},U_{p-1})_G$ is $(\epsilon,d)$-super-regular, then there is at least one edge
$(v_{p\ell-2},v_{p\ell-1})$ from $X$ to $Y$ and $v_1\cdots v_{p\ell}$ with $v_{p\ell}=y$ is the desired path $P$ from $x$
to $y$ of length $p\ell$.

We start the proof of Lemma \ref{path} by recalling the blowup lemma of Koml\'os, S\'ark\"ozy and Szemer\'edi \cite{KSS}.

\begin{lemma}\label{blowup}
Given a graph $R$ of order $k$ and parameters $d, \Delta>0$,
there exists an $\eta_0 = \eta_0(d,\Delta,k) >0$ such that whenever
$0 < \eta \leq \eta_0$, the following holds.
Let $V_1,\cdots,V_k$ be disjoint sets and let $R^*$ be the
graph on $V_1 \cup \cdots \cup V_k$ obtained by replacing each
edge $ij$ of $R$ by the complete bipartite graph between $V_i$ and $V_j$.
Let $G$ be a spanning subgraph of $R^*$ such that for each edge $ij$ of $R$
the bipartite subgraph of $G$ consisting of all edges between $V_i$ and $V_j$
is $(\eta,d)$-super-regular. Then $G$ contains a copy of every subgraph $H$ of
$R^*$ with maximum degree $\Delta(H)\le \Delta$. Moreover, this copy of $H$ in $G$
maps the vertices of $H$ to the same sets $V_i$ as the copy of $H$ in $R^*$,
i.e., if $h \in V(H)$ is mapped to $V_i$ by the copy of $H$ in $R^*$,
then it is also mapped to $V_i$ by the copy of $H$ in $G$.
\end{lemma}

From the blowup lemma, we can quickly deduce the following lemma.

\begin{lemma}\label{linking}
For every $0<d<1$ there is $\eps_0>0$ so that the following holds for $0<\eps<\eps_0$.
Suppose $p \ge 4$, let $U_1,\cdots,U_p$ be pairwise disjoint sets of size $n$, for some $n$,
and suppose $G$ is a graph on $U_1 \cup \cdots \cup U_p$ such that each pair
$(U_i,U_{i+1})$, $1 \le i \le p-1$ is $(\eps,d)$-super-regular.
Let $f:U_1 \to U_p$ be any bijective map.
Then there are $n$ vertex-disjoint paths from $U_1$ to $U_p$ so that
for every $x \in U_1$ the path starting from $x$ ends at $f(x) \in U_p$.
\end{lemma}

\begin{proof}
Choose a sequence $1=i_1<i_2<\cdots<i_t=p$ so that $3 \le i_j - i_{j-1} \le 5$ for $2 \le j \le t$.
Let $f_j:U_{i_{j-1}}\to U_{i_j}$ be any bijective maps with $f = f_t \circ \cdots \circ f_2$. Let $G_j$ be the graph
obtained from the restriction of $G$ to $U_{i_{j-1}} \cup U_{i_{j-1}+1} \cup \ldots \cup U_{i_j}$ by identifying
each vertex $x \in U_{i_{j-1}}$ with $f_j(x) \in U_{i_j}$. By Lemma \ref{blowup} we can find $n$ vertex-disjoint
cycles in $G_j$ of length $i_j-i_{j-1}$, provided that $\eps_0 < \eta(d,2,i_j-i_{j-1})$,
which only depends on $d$ as $i_j-i_{j-1} \leq 5$. These $n$ cycles
correspond to $n$ vertex-disjoint paths in $G$ from $U_{i_{j-1}}$ to $U_{i_j}$, such that for every $x \in U_{i_{j-1}}$,
the path starting from $x$ ends at $f_j(x) \in U_{i_j}$. By concatenating these paths, we get the desired $n$ vertex-disjoint
paths from $U_1$ to $U_p$ so that for every $x \in U_1$ the path starting from $x$ ends at $f(x) \in U_p$.
\end{proof}

Now we give the proof of Lemma \ref{path}.

\noindent {\bf Proof of Lemma \ref{path}.} Suppose $G$ is a digraph on $U_1 \cup \cdots \cup U_p$, where $|U_i|=n$, $1 \le i \le p$, such that each
$(U_i,U_{i+1})_G$ is $(\eps,d)$-super-regular, with $\eps<\eps_0$ given by Lemma \ref{linking}.
Suppose also $x \in U_1$, $y \in U_p$ and $1 \le \ell \le n$.
We need to find a path $P$ of length $p\ell$ from $x$ to $y$.
First we apply the blowup lemma to find a perfect matching from $U_p \sm y$ to $U_1 \sm x$.
We label $U_1$ as $\{x_1,\cdots,x_n\}$ and $U_p$ as $\{y_1,\cdots,y_n\}$
with $x_1=x$ and $y_1=y$, so that the matching edges go from $y_i$ to $x_i$ for $2 \le i \le n$.
Then we apply Lemma \ref{linking} to find $n$ vertex-disjoint paths from $U_1$ to $U_p$
so that the path $P_i$ starting at $x_i$ ends at $y_{i+1}$ for $1 \le i \le \ell-1$
and the path $P_\ell$ starting at $x_\ell$ ends at $y_1=y$ (the other paths can be arbitrary).
Now our required path $P$ is $x_1 P_1 y_2 x_2 P_2 y_3 \cdots x_\ell P_\ell y_1$. \qed

We finish the section with two simple lemmas concerning super-regularity. The first lemma tells us that large induced subgraphs
of super-regular bipartite graphs are also super-regular.

\begin{lemma}\label{prevsubsuper} Let $G$ be a bipartite graph with parts $A$ and $B$ that is $(\epsilon,d)$-super-regular,
$\eps < 1/2 \le \alpha <1$, $A' \subset A$ and $B' \subset B$ with $|A'|/|A|, |B'|/|B| \ge \alpha $, and $G'$
be the induced subgraph of $G$ with parts $A'$ and $B'$. Then $G'$ is $(2\epsilon,d-1+\alpha)$-super-regular.
\end{lemma}
\begin{proof}
Super-regularity of $G$ implies that each vertex $a \in A' \subset A$ satisfies $d_G(a) \geq (d-\epsilon)|B|$. Hence,
$$d_{G'}(a) \geq d_{G}(a)-(|B|-|B'|) \geq (d-\epsilon)|B|-(|B|-|B'|) \geq (d-(1-\alpha)-\epsilon)|B| \geq (d-(1-\alpha)-\epsilon)|B'|.$$ Likewise,
each vertex $b \in B'$ satisfies $d_{G'}(b) \geq (d-(1-\alpha)-\epsilon)|B'|$.

Let $X \subset A'$ and $Y \subset B'$ with $|X| > 2\eps|A'|$ and $|Y| > 2\eps|B'|$. Since
$1/2 \le \alpha \le |A'|/|A|, |B'|/|B|$ we have $|X| > \eps |A|$ and $|Y| > \eps |B|$.
Now the pair $(A,B)_G$ is $\eps$-regular,
so $|d(X,Y)-d(A,B)| < \eps$, and the triangle inequality gives
$$|d(X,Y)-d(A',B')| \leq |d(X,Y)-d(A,B)|+|d(A,B)-d(A',B')| < 2\epsilon.$$
Hence, $G'$ is $(2\epsilon,d-1+\alpha)$-super-regular.
\end{proof}

For any bounded degree subgraph $H$ of a reduced graph $R$,
the next lemma allows us to make the pairs $(V_i,V_j)_G$ corresponding
to edges $ij$ of $H$ super-regular by deleting a few vertices from each $V_i$.

\begin{lemma}\label{boundeddegH}
Suppose $R$ is the reduced digraph with parameters $(\epsilon,d)$
of a Szemer\'edi partition $V_G=V_0 \cup V_1 \cup \ldots \cup V_k$ of a digraph $G$
and $H$ is a subdigraph of $R$ with maximum total degree at most $\Delta$,
where $\Delta \leq \frac{1}{2\epsilon}$. Then for each $i$, $1 \leq i \leq k$, there is
$U_i \subset V_i$ with $|U_i| =(1-\Delta\epsilon)|V_i|$ such that
for each edge $ij$ of $H$, the pair $(U_i,U_j)_G$ is $(2\epsilon,d-\Delta\epsilon)$-super-regular.
\end{lemma}
\begin{proof}
For each edge $ij$ of $H$, delete all vertices in $V_i$ with less than $(d-\epsilon)|V_j|$ outneighbours
in $V_j$ and all vertices in $V_j$ with less than $(d-\epsilon)|V_i|$ inneighbours in $V_i$. For each edge $ij$ of $H$,
less than $\epsilon|V_i|$ elements are deleted from $V_i$ and less than $\epsilon|V_j|$ elements are deleted from $V_j$.
Indeed, if the subset $S \subset V_i$ of vertices with less than $(d-\epsilon)|V_j|$ outneighbours in $V_j$ has cardinality
$|S| \geq \epsilon |V_i|$, then $d_G(S,V_j)<d-\epsilon$,
in contradiction to $ij$ being an edge of the reduced graph $R$.
Likewise, at most $\epsilon|V_j|$ elements are deleted from $V_j$
for each edge $ij$. Hence, in total,
at most $\Delta\epsilon|V_i|$ vertices are deleted from each $V_i$. Delete further vertices from each $V_i$ until the resulting
subset $U_i$ has cardinality $(1-\Delta\epsilon)|V_i|$. 
For each edge $ij$ of $H$, each vertex in
$U_i$ has at least $(d-\epsilon)|V_j|$ outneighbours in $V_j$ and hence at least
$$(d-\epsilon)|V_j|-(|V_j|-|U_j|)=(d-(\Delta+1)\epsilon)|V_j| \geq (d-(\Delta+1)\epsilon)|U_j|$$ outneighbours in $U_j$.
Similarly, for each edge $ij$ of $H$, each vertex in $V_j$ has at least $(d-(\Delta+1)\epsilon)|U_i|$ inneighbours in $U_i$.
Letting $\alpha=|U_i|/|V_i|=1-\Delta\epsilon$, we have $\alpha \geq 1/2$.
For each edge $ij$ of $H$, since $(V_i,V_j)_G$ is $\epsilon$-regular, Lemma \ref{prevsubsuper} implies
that $(U_i,U_j)_G$ is $2\epsilon$-regular and hence is $(2\epsilon,d-\Delta\epsilon)$-super-regular.
\end{proof}

\section{Cycles of almost given length}

Now we will apply the regularity lemma and Corollary \ref{beta}
to answer the question of Yuster mentioned in the introduction.

{\bf Proof of Theorem \ref{given-length}:}
Choose parameters $0 < 1/n_0 \ll 1/M \ll \eps \ll d \ll \delta, \theta$ and $M'=\eps^{-1}$.
Suppose $G$ is a digraph on $n \ge n_0$ vertices with $\beta(G) \ge \theta n^2$. Note that $\theta < 1/2$
as in any linear ordering of the vertices of $G$, deleting all the forward edges or all the backward edges yields an
acyclic digraph. Apply Lemma \ref{direg} to obtain a partition of the vertices of $G$ into
$V_0,V_1,\cdots,V_k$ for some $M' \le k \le M$
and let $R$ be the reduced graph on $[k]$ with parameters $(\eps,d)$.
As noted in the previous section, there are at most $2d n^2$
edges of $G$ that do not belong to $E_G(V_i,V_j)$ for some edge $ij \in R$.
We can make $G$ acyclic by deleting these edges
and at most $\beta(R)(n/k)^2$ edges corresponding to edges of $R$,
so we must have $\beta(R) \ge (\theta-2d)k^2$.
Let $S_1,\cdots,S_g$ be the strong components of $R$ and
suppose $\beta(S_i)=\theta_i|S_i|^2$. Then $\sum_{i=1}^g |S_i|=k$
and $\sum_{i=1}^g \theta_i|S_i|^2 = \sum_{i=1}^g \beta(S_i) = \beta(R)
\ge (\theta-2d)k^2$. It follows that we can choose some $S_j$ with
$\theta_j |S_j| \ge (\theta-2d)k$ (otherwise we would have
$\sum_{i=1}^g \theta_i|S_i|^2 < (\theta-2d)k\sum|S_i| = (\theta-2d)k^2$).

Next we restrict our attention to $S_j$ and
repeatedly delete any vertex with outdegree less than $\theta_j|S_j|$ in $S_j$.
\COMMENT{
1. Could do indegree as well.
2. Find a better argument not to lose factor of $2$ and get tight bound?
}
We must arrive at some graph $R_0$ on $k_0 \le |S_j|$ vertices with minimum outdegree
at least $\theta_j|S_j| \ge (\theta-2d)k$ and
$\beta(R_0) \ge \theta_j |S_j|k_0$. Indeed,
otherwise we could make $S_j$ acyclic by deleting less than
$\theta_j |S_j|k_0 + (|S_j|-k_0)\theta_j|S_j| = \theta_j |S_j|^2$ edges, which is impossible.
Let $C = c_1 \cdots c_p$ be a directed cycle in $R_0$ of length $p \ge (\theta-2d)k$. It
 can be found by considering a longest directed path and using the fact that
the end of the path has at least $(\theta-2d)k$ outneighbours, which all lie on the path.
Recall that $$\beta(S_j) = \theta_j |S_j|^2 \geq (\theta-2d)k|S_j| \geq (\theta-2d)|S_j|^2.$$
By Corollary \ref{beta}, if $S_j$ is $r$-free, then $(\theta-2d)|S_j|^2 \leq \beta(S_j) \leq 25|S_j|^2/r^2,$ so
$$r \leq 5(\theta-2d)^{-1/2}<(5+\delta)\theta^{-1/2},$$ where we use $d \ll \delta,\theta$.
Therefore, there is a directed cycle $C'=c'_1 \cdots c'_r$ in $S_j$
of length $r$ for some $2 \le r \le (5+\delta)\theta^{-1/2}$
(which may intersect $C$ in an arbitrary fashion).
Also, by strong connectivity of $S_j$ we can find a directed path $Q_1$ from $c_p$ to $c'_r$
and a directed path $Q_2$ from $c'_r$ to $c_1$. Suppose that the lengths of these
paths are respectively $q_1$ and $q_2$. We note that $q_1,q_2 \le k$.

Let $H$ denote the digraph with vertex set $V_{S_j}$ and edge set $E_{C} \cup E_{C'} \cup E_{Q_1} \cup E_{Q_2}$.
Note that the maximum total degree of $H$ is at most $8$ as each path and cycle has maximum total degree at most $2$.
By Lemma \ref{boundeddegH}, for each vertex $i$ of $S_j$ there is
$U_i \subset V_i$ with $|U_i| =(1-8\epsilon)|V_i|$ such that for each edge $ij$ of $H$,
the pair $(U_i,U_j)_G$ is $(2\epsilon,d-8\epsilon)$-super-regular.

Suppose $0 \le m \le (1-\delta)\theta n$ is given.
We give separate arguments depending on whether the cycles we seek in $G$ are short or long.
First consider the case $m < 3k$. Choose $\ell$ divisible by $r$ with $m \le \ell < m+r$.
Then we can find a cycle of length $\ell$ within the classes $U_i$
corresponding to $C'$, as noted after Lemma \ref{path}. (This argument holds as long as $r \geq 4$ or $\ell \geq 2r$.
If otherwise, then $\ell=r \in \{2,3\}$ and we can find a cycle of length $2r$ in $G$. This $2r$-cycle completes
this case as $m \leq \ell = r \le 2r \leq 6 <5\theta^{-1/2}$, where we use $\theta < 1/2$.)
Now suppose $m \ge 3k$ and write
$m=q_1+q_2+sp+t$, with $0 \le t < p$ and $1 \le s < (1-\delta/2)n/k$
(since $p \ge (\theta-2d)k$). The integer $t$ is indeed nonnegative since $q_1,q_2,p \le k$ and $m \geq 3k$.
We can choose $\ell = q_1+q_2+sp+u$
where $u < p+r$ is a multiple of $r$ and $m \le \ell < m+r$. Say that a path $P = v_1 \cdots v_e$ in $G$ corresponds to a walk $W = w_1 \cdots w_e$ in $R$
if every edge $v_i v_{i+1}$, $1 \le i \le e-1$ of $P$ goes from $U_{w_i}$ to $U_{w_{i+1}}$.
For $ij$ an edge of $H$, the pair $(U_i,U_j)_G$ is $(2\epsilon,d-8\epsilon)$-super-regular, so any vertex in $U_i$ has
at least $(d-10\epsilon)|U_j|$ outneighbours in $U_j$. Therefore, we can greedily find
\begin{enumerate}
\item
a directed path $P_1$ in $G$ corresponding to $Q_1$ in $R$,
starting at some $y \in U_{c_p}$ and ending at some $z \in U_{c'_r}$,
\item
a directed path $P_2$ in $G$ corresponding to $u/r$ copies of $C'$ in $R$,
starting at $z$ and ending at some other $z' \in U_{c'_r}$,
and avoiding $P_1$,
\item
a directed path $P_3$ in $G$ corresponding to $Q_2$ in $R$,
starting at $z'$ and ending at some $x \in U_{c_1}$, avoiding $P_1 \cup P_2$.
\end{enumerate}
Let $P$ be the path $P_1P_2P_3$. Note that $P$ has at most $u/r+2$ vertices in each $U_i$.
As we next find a path from $x$ to $y$ disjoint from $P \sm \{x,y\}$,
we delete the vertices of $P \sm \{x,y\}$
and also at most $u/r+2$ vertices from each $U_i$
so that they all still have the same size, letting $U_i'$ be the resulting subset of $U_i$. Now
$$|U_i'| \ge |U_i| -(u/r+2) \ge (1-8\epsilon)|V_i|-(u/r+2) > (1-d/2)|V_i|>(1-\delta/2)(n/k) = s.$$
This also gives $|U_i'|/|U_i|>1-d/2$ for each vertex $i$ of $S_j$. For each edge $ij$ of $H$, $(U_i,U_j)$
is $(2\epsilon,d-8\epsilon)$-super-regular. Hence, Lemma \ref{prevsubsuper} with $\alpha=1-d/2$ implies that
$(U_i',U_j')$ is $(4\epsilon,d/4)$ super-regular, as $d-8\epsilon-d/2=d/2-8\epsilon \ge d/4$.
Therefore, we can apply Lemma \ref{path} with $U_i = U_{c_i}'$, $1 \le i \le p$
to obtain a directed path from $x$ to $y$ of length $sp$.
Combining this with the path $P$ already found from $y$ to $x$
gives a directed cycle of length $\ell$, as required.\qed

For the proof of Theorem \ref{periodic} we need the following
two facts from elementary number theory.
\COMMENT{
\item[Ear decomposition] (NOT NEEDED)
Every strongly connected digraph $G$ on at least $2$ vertices has
an `ear decomposition' $P_0,P_1,\cdots,P_t$ for some $t$,
where $P_0$ is a directed cycle, each $P_i$ is either a directed path
or a directed cycle, the edges of the $P_i$ partition the edges of $G$,
and, for each $1 \le i \le t$, either $P_i$ is a cycle with exactly one vertex in
$G_{i-1}:=\cup_{j\le i-1} P_j$ or $P_i$ is a directed path that starts and ends in
$G_{i-1}$ but contains no other vertex of $G_{i-1}$.
(For a proof see Theorem 7.2.2 in \cite{BJG}.)
}
\begin{description}
\item[Chinese Remainder Theorem.]
Suppose $x_1,\cdots,x_t$ are integers with greatest common factor $1$.
Then any integer $n$ can be expressed as $n=a_1x_1 + \cdots + a_tx_t$
with integers $a_1,\cdots,a_t$.
\item[Sylvester's `coin problem'.]
Suppose $x$ and $y$ are coprime positive integers.
Then every integer $n \ge (x-1)(y-1)$ can be
represented as $n=ax+by$ with $a$, $b$ non-negative integers.
\end{description}

\nib{Proof of Theorem \ref{periodic}.}
It is straightforward to see that $\lambda$ (similarly to $\beta$)
is additive on strong components, i.e., if a digraph $G$ has strong components
$T_1,\ldots,T_g$, then $\lambda(G)=\sum_{i=1}^g\lambda(T_i)$. Also, $\lambda(G) \leq \beta(G)$, since every acyclic digraph
is pseudoperiodic. Therefore we start as in the proof of Theorem \ref{given-length}
by applying Lemma \ref{direg} to obtain a partition of the vertices of $G$ into
$V_0,V_1,\cdots,V_k$ for some $M' \le k \le M$
and letting $R$ be the reduced graph on $[k]$ with parameters $(\eps,d)$.
As before  $G$ has at most $2d n^2$ edges not corresponding
to edges of the reduced graph $R$, so we must have $\lambda(R) \ge (\theta-2d)k^2$.
Then, as in the proof of Theorem \ref{given-length},  we find a strong component $S_j$ of $R$ with
$\beta(S_j) \ge \lambda(S_j)=\theta_j|S_j|^2$ and $\theta_j |S_j| \ge (\theta-2d)k$,
directed cycles $C = c_1 \cdots c_p$ and $C'=c'_1 \cdots c'_r$ in $S_j$
with $p \ge (\theta-2d)k$ and $2 \le r \le (5+\delta)\theta^{-1/2}$,
a directed path $Q_1$ from $c_p$ to $c'_r$ of length $q_1 \le k$
and a directed path $Q_2$ from $c'_r$ to $c_1$ of length $q_2 \le k$.

Next we show how to construct a closed walk $W$ in $S_j$ starting and ending at $c'_r$
with length $l(W)=w$ coprime to $r$.
Since $S_j$ is not $f$-periodic for any $f \ge 2$,
for each prime factor $f$ of $r$ there is a directed cycle with length not divisible by $f$.
Therefore we can choose cycles $D_1,\cdots,D_r$
so that $l(C'),l(D_1),\cdots,l(D_r)$ have greatest common factor $1$.
Fix vertices $d_i \in D_i$, $1 \le i \le r$
and choose the following directed paths in $S_j$
(which exist by strong connectivity):
$Q'_1$ from $c'_r$ to $d_1$ and $Q''_1$ from $d_1$ to $c'_r$,
$Q'_i$ from $d_{i-1}$ to $d_i$ and $Q''_i$ from $d_i$ to $d_{i-1}$ for $2 \le i \le r$.
Let $W'$ be the walk $Q'_1 \cdots Q'_r Q''_r \cdots Q''_1$.
By the Chinese Remainder Theorem we can find integers $a_1,\cdots,a_r$ such that
$l(W') + a_1 l(D_1) + \cdots a_r l(D_r) \equiv 1$ mod $r$.
By reducing mod $r$ we can assume that $0 \le a_i \le r-1$ for $1 \le i \le r$.
We let $W$ be the walk obtained from $W'$ by including $a_i$ copies of $D_i$
when $d_i$ is first visited. That is, we obtain $W$ by walking along $W'$,
and, for $1 \leq i \leq r$, when we first reach $d_i$, before we continue onto the next vertex,
we first walk $a_i$ times around the cycle $D_i$.
Then $l(W)\equiv 1$ mod $r$ is coprime to $r$.
The walk $W$ visits any vertex at most $2r^2$ times. Indeed, each of the $2r$ directed paths $Q'_i$ and $Q''_i$ visit
each vertex at most once, and each time we go around cycle $D_i$ adds at most one new visit to any vertex, so $W$
visits each vertex at most $2r+a_1+\ldots+a_r \leq 2r+r^2 \le 2r^2$ times. As $S_j$ has at most $k$ vertices and
visits each vertex at most $2r^2$ times,
$w=l(W) \le 2r^2 k$.
\COMMENT{
Careful! Many natural seeming arguments for this claim fail.
At first I thought one can get every vertex visited at most twice:
if say $W$ visits $3$ times then let $a$ and $b$ be the lengths of two
different closed subwalks of $W$ starting at ending at $v$,
and consider replacing $w$ with $w-a$, $w-b$ and $w-a-b$... this doesn't work.
I also had a more elaborate approach based on ear decompositions,
but this is simpler.
}

Let $H$ be the digraph with vertex set $V_{S_j}$ and edge set $E_C \cup E_{C'} \cup E_{W}$. Since $W$ visits any vertex
at most $2r^2$ times, each vertex in $W$ is in at most $4r^2$ edges of $H$. Therefore, $H$ has maximum total
degree at most $4+4r^2 \leq 8r^2$. By Lemma \ref{boundeddegH}, for each vertex $i$ of $S_j$ there is
$U_i \subset V_i$ with $|U_i| =(1-8r^2\epsilon)|V_i|$ such that for each edge $ij$ of $H$,
the pair $(U_i,U_j)_G$ is $(2\epsilon,d-8r^2\epsilon)$-super-regular.

Fix any $\ell$ with $500\theta^{-3/2}M \le \ell \le (1-\delta)\theta n$.
We will show that $G$ contains a directed cycle of length $\ell$.
As $2 \leq r < 6\theta^{-1/2}$, $p,q_1,q_2 \leq k \leq M$ and $w \leq 2r^2k$, we have
$$\ell \geq 500\theta^{-3/2}M
\geq 3k+2r^3k \geq q_1+q_2+p+rw.$$
Therefore, we can
write $\ell = q_1+q_2+sp+u$, with $rw \le u < rw + p$ and $1 \le s < (1-\delta/2)n/k$
(the last inequality uses $p \geq (\theta-2d)k$).
Since $r,w$ are coprime, by the `coin problem' result of Sylvester
we can write $u=ar+bw$ with $a$, $b$ non-negative integers. We have $a \leq u/r < w+p \leq 2r^2k+k$ and
$b \leq u/w <r+p \leq 2k$. For $ij$ an edge of $H$, the pair $(U_i,U_j)_G$ is $(2\epsilon,d-8r^2\epsilon)$-super-regular, so any vertex in $U_i$ has
at least $(d-10r^2\epsilon)|U_j|$ outneighbours in $U_j$.
Therefore, we can greedily find
\begin{enumerate}
\item
a directed path $P_1$ in $G$ corresponding to $Q_1$ in $R$,
starting at some $y \in U_{c_p}$ and ending at some $z \in U_{c'_r}$,
\item
a directed path $P_2$ in $G$ corresponding to $a$ copies of $C'$ in $R$,
starting at $z$ and ending at some other $z' \in U_{c'_r}$,
and avoiding $P_1$,
\item
a directed path $P_3$ in $G$ corresponding to $b$ copies of $W$ in $R$,
starting at $z'$ and ending at some other $z'' \in U_{c'_r}$,
and avoiding $P_1 \cup P_2$,
\item
a directed path $P_4$ in $G$ corresponding to $Q_2$ in $R$,
starting at $z''$ and ending at some $x \in U_{c_1}$, avoiding $P_1 \cup P_2 \cup P_3$.
\end{enumerate}
Let $P$ be the path $P_1P_2P_3P_4$. As we walk along path $P$, for each $i$,
the number of times $U_i$ is visited is at most once for $P_1$, at most $a$ times for $P_2$,
at most $b \cdot 2r^2$ times for $P_3$,
and at most once for $P_4$. Therefore, for each $i$, $$|P \cap U_i| \leq 1+a+b \cdot 2r^2+1 \leq
1+2r^2k+k+2k \cdot 2r^2+1 \leq 10r^2k.$$

We delete the vertices of $P \sm \{x,y\}$ as we next find a directed
path from $x$ to $y$ that is disjoint from $P \sm \{x,y\}$. We further delete at most
$10r^2k$ vertices from each $U_i$ so that they all still have the same size,
and let $U_i'$ be the resulting subset of $U_i$. Now
$$|U_i'| \ge |U_i| - 10r^2k = (1-8r^2\epsilon)|V_i| - 10r^2k > (1-d/2)|V_i|>(1-\delta/2)(n/k) = s.$$
Then $|U_i'|/|U_i|>(1-d/2)$, and Lemma \ref{prevsubsuper} with $\alpha=1-d/2$ implies that
each pair $(U_i',U_j')_G$ with $ij$ an edge of $H$ is $(4\epsilon,d/4)$-super-regular, as $d-8r^2\epsilon-d/2 \geq d/4$.
Therefore we can apply Lemma \ref{path} with $U_i = U_{c_i}'$, $1 \le i \le p$
to obtain a directed path from $x$ to $y$ of length $sp$.
Combining this with the path $P$ already found from $y$ to $x$
gives a directed cycle of length $\ell$, as required. \qed

\section{Concluding remarks}

\begin{itemize}

\item
We have not presented the best possible constants that come from our methods,
opting to give reasonable constants that can be obtained with relatively clean proofs.
With more work one can replace the constant $25$ in Theorem \ref{mu}, and so in Corollary \ref{beta},
by a constant that approaches $8$ as $r$ becomes large. However, Sullivan \cite{S} conjectures
that the correct constant is $2$, and it would be interesting to close this gap.
The problems of estimating $\beta$ and $\mu$ are roughly equivalent:
we used the bound on $\mu$ from Theorem \ref{mu} to establish
 the bound on $\beta$ in Theorem \ref{beta}.
Conversely, if we delete $\beta(G)$ edges from $G$ to make it acyclic, order the vertices
so that all remaining edges point in one direction and take $S$ to be the first $n/2$ vertices
in the ordering we see that $$\mu(G)(n/2) = \mu(G)|S| \leq \mu(S)|S| =\min(e(S,V_G \setminus S),e(V_G \setminus S,S)) \leq
\beta(G),$$ so a bound on $\beta$ gives a bound on $\mu$.
However, these arguments may be too crude to give the correct constants.

\item
Applying this better constant $8$ (mentioned above) in Corollary \ref{beta} we can replace the constant $5$ by $3$ (say)
in Theorem \ref{given-length}, so that the parameter $K$ in Yuster's question
(the length of the interval where we look for a cycle length) is determined up
to a factor of $3$. The parameter $\eta$ (the maximum length of a cycle as a proportion of $n$)
is determined up to a factor of about $4$ if the
question is posed for oriented graphs, or a factor $2$ if the question is posed for digraphs. Indeed,
Yuster shows that $\eta \le 4\theta$ for oriented graphs by taking $1/4\theta$
copies of a random regular tournament on $4\theta n$ vertices;
for digraphs one can show $\eta \le 2\theta$ by taking $1/2\theta$ copies of
the complete digraph on $2\theta n$ vertices.
We can find longer cycles in a periodic digraph $G$ on $n$ vertices with $\beta(G) \ge \theta n^2$,
but $\theta n$ is still the correct bound up to a constant of about $2$,
as may be seen from the blowup of a $2$-cycle with
parts of size $(1+2\theta)\theta n$ and $(1-(1+2\theta)\theta)n$.
\COMMENT{
Avoid losing factor $2$ in vertex deletion argument from $R_0$ to $R$
and get correct constant for $\eta$?
}

\item
If a digraph $G$ is far from being acyclic but we can obtain a pseudoperiodic digraph $G'$ by deleting few edges of $G$,
then some strong component of $G'$ has small period. More precisely,
if $\beta(G) \ge \theta n^2$ and we can obtain a pseudoperiodic $G'$
by deleting at most $\delta n^2$ edges from $G$ then some strong component of $G'$
must have period at most $(\theta-\delta)^{-1/2}$. To see this,
note that $\beta(G') \ge (\theta-\delta)n^2$, so some strong component $H$
of $G'$ satisfies $\beta(H) \ge (\theta-\delta)m^2$, where $m=|V_H|$.
Since $G'$ is pseudoperiodic $H$ is $p$-periodic, for some $p$,
so is contained in the blowup of a $p$-cycle,
i.e. the vertex set of $H$ can be partitioned as $V(H) = V_1 \cup \cdots \cup V_p$
so that every edge goes from $V_i$ to $V_{i+1}$, for some $1 \le i \le p$,
writing $V_{p+1}=V_1$. (For a proof see Theorem 10.5.1 in \cite{BJG}.)
Write $t_i = |V_i|/m$. Then there is
some $1 \le i \le p$ for which $t_it_{i+1} \le 1/p^2$.
This can be seen from the arithmetic-geometric mean inequality: we have
$1 = \sum_{i=1}^p t_i \ge p \prod_{i=1}^p t_i^{1/p} = p \prod_{i=1}^p (t_it_{i+1})^{1/2p}$,
so $\prod_{i=1}^p t_it_{i+1} \le (1/p^2)^p$. It follows that $\beta(H) \le (m/p)^2$,
i.e. $p \le  (\theta-\delta)^{-1/2}$, as required.
\COMMENT{
Analyse periodic case to get correct constant for $K$?
}

\item The dependence of $C$ on $\theta$ which we get in  Theorem \ref{periodic} is quite poor since the proof uses
Szemer\'edi's regularity lemma and the value of $C$ depends on the number of parts in the regular partition. It would
be interesting to determine the right dependence of $C$ on $\theta$. One should note that we obtained good
constants in the proof of Theorem 1.4 despite using the regularity lemma, 
so it may not be necessary to avoid its use.
\end{itemize}

\end{document}